\documentclass[12pt,twoside,reqno]{amsart}
\usepackage{amsmath}
\usepackage{amsfonts}
\usepackage{amssymb}
\usepackage{color}
\usepackage{mathrsfs}
\usepackage{cleveref}
\usepackage{cite}
\usepackage{geometry}
\usepackage{marginnote}
\usepackage{todonotes}
\allowdisplaybreaks
\newtheorem{theorem}{Theorem}[section]
\newtheorem{corollary}[theorem]{Corollary}
\newtheorem{lemma}[theorem]{Lemma}
\newtheorem{proposition}[theorem]{Proposition}
\newtheorem{definition}[theorem]{Definition}
\newtheorem{remark}[theorem]{Remark}
\allowdisplaybreaks
\numberwithin{equation}{section}
\begin{document}
\title{The Redner--ben-Avraham--Kahng cluster system without growth condition on the kinetic coefficients} 
\author{Philippe Lauren\c{c}ot}
\address{Laboratoire de Math\'ematiques (LAMA) UMR~5127, Universit\'e Savoie Mont Blanc, CNRS, F--73000 Chamb\'ery, France}
\email{philippe.laurencot@univ-smb.fr}


\date{\today}

\begin{abstract}
Existence of global mild solutions to the infinite dimensional Redner--ben-Avraham--Kahng cluster system is shown without growth or structure condition on the kinetic coefficients, thereby extending previous results in the literature. The key idea is to exploit the dissipative features of the system to derive a control on the tails of the infinite sums involved in the reaction terms. Classical solutions are also constructed for a suitable class of kinetic coefficients and initial conditions.
\end{abstract}

\maketitle

%
%
\pagestyle{myheadings}
\markboth{\sc{Ph. Lauren\c cot}}{\sc{The Redner-ben-Avraham-Kahng cluster system without growth condition}}

\section{Introduction}

The aim of this note is to investigate the existence of global mild solutions to a one species annihilation cluster system introduced in \cite{RbAK1987, dCPS2012} and referred to as the `cluster eating' system. This model describes the evolution of a set of clusters, each cluster being characterized by a single parameter $i\in\mathbb{N}\setminus\{0\}$ accounting for the number of active sites it bears. Denoting clusters bearing $i$ active sites by $P_i$, $i\ge 1$, the dynamics is governed by pairwise encounters between incoming clusters $P_i$ and $P_j$ resulting in the annihilation reaction $P_i+P_j \to P_{|i-j|}$ with no product formed when $i=j$. The number density $f_i=f_i(t)\ge 0$ of clusters with $i$ active sites, $i\ge 1$, at time $t\ge 0$, then evolves according to
\begin{subequations}\label{rbak}
	\begin{align}
		\frac{\mathrm{d}f_i}{\mathrm{d}t} & = \sum_{j=1}^\infty a_{i+j,j} f_{i+j} f_j - \sum_{j=1}^\infty a_{i,j} f_i f_j\,, \qquad i\ge 1\,, \label{rbak1} \\
		f_i(0) & = f_i^{in}\,, \qquad i\ge 1\,, \label{rbak2}
	\end{align}
\end{subequations}
where $a_{i,j}$ denotes the reaction rate between incoming clusters with respective sizes $i\ge 1$ and $j\ge 1$ and satisfies
\begin{equation}
	a_{i,j} = a_{j,i}\ge 0\,, \qquad i,j\ge 1\,. \label{a}
\end{equation}
The cluster system~\eqref{rbak} predicting the dynamics of $f=(f_i)_{i\ge 1}$ is a countably infinite system of quadratic differential equations which are strongly coupled due to the infinite series involved in the reaction terms. This structure actually prevents the use of the classical theory of ordinary differential equations to study the well-posedness of~\eqref{rbak}. Nevertheless, the infinite system~\eqref{rbak} is, at least formally, the limit as $n\to\infty$ of the finite dimensional cluster system \cite{RbAK1987}
\begin{subequations}\label{rbakf}
	\begin{align}
		\frac{\mathrm{d}f_i^n}{\mathrm{d}t} & = \sum_{j=1}^{n-i} a_{i+j,j} f_{i+j}^n  f_j^n - \sum_{j=1}^n a_{i,j} f_i^n f_j^n\,, \qquad 1\le i\le n\,, \label{rbakfa} \\
		f_i^n(0) & = f_i^{in}\,, \qquad 1\le i\le n\,, \label{rbakfb}
	\end{align}
\end{subequations}
where $n\ge 3$. One may then expect to obtain solutions to~\eqref{rbak} as limits of solutions $f^n=(f_i^{n})_{1\le i \le n}$ as $n\to\infty$ and this approach has proved successful when the rate coefficients grow at most quadratically; that is, 
\begin{equation}
	\sup_{i,j\ge 1}\left\{ \frac{a_{i,j}}{ij} \right\} < \infty\,, \qquad i,j\ge 1\,.\label{ga}
\end{equation} 
In that case, the existence of classical solutions is shown in \cite{dCPS2012} for $f^{in}=(f_i^{in})_{i\ge 1}\in X_{1,+}$, where the Banach space $X_m$ is defined for $m\in\mathbb{R}$ by
\begin{equation*}
	X_m := \{ z=(z_i)_{i\ge 1}\ :\ \|z\|_m := \sum_{i=1}^\infty i^m |z_i|<\infty \}\,, 
\end{equation*}
and $X_{m,+}$ denotes its positive cone. Well-posedness in $X_{1,+}$ is also established in \cite{dCPS2012} under the stronger growth condition $\sup_{i,j\ge 1}\left\{ a_{i,j}/\sqrt{ij}\right\}<\infty$. The same approach is used in \cite{Ver2024} to construct mild solutions to~\eqref{rbak} when the rate coefficients have the following structure: there are a sequence $(r_i)_{i\ge 1}$ of positive real numbers, a family $(\alpha_{i,j})_{i,j\ge 1}$ of non-negative numbers, and $R>0$ such that 
\begin{equation}
	a_{i,j} = r_i r_j + \alpha_{i,j}\,, \quad r_i \ge R i\,, \quad i,j\ge 1\,, \qquad \sup_{i,j\ge 1}\left\{ \frac{\alpha_{i,j}}{r_i r_j} \right\}< \infty\,. \label{sa}
\end{equation}
Note that no growth condition on the sequence $(r_i)_{i\ge 1}$ is required in~\eqref{sa}, in contrast to~\eqref{ga}. 

We shall actually prove that, given $f^{in}\in X_{1,+}$, there is a global mild solution to~\eqref{rbak} under the sole non-negativity and symmetry assumption~\eqref{a} on the rate coefficients $(a_{i,j})_{i,j\ge 1}$. Before stating the existence result, let us first recall the definition of a mild solution to~\eqref{rbak} in $X_{1,+}$.

\begin{definition}\label{defms}
	Consider $f^{in} = (f_i^{in})_{i\ge 1}\in X_{1,+}$. A mild solution $f=(f_i)_{i\ge 1}$ to~\eqref{rbak} is a sequence of non-negative functions satisfying
	\begin{itemize}
		\item [(a1)] $f\in L^\infty((0,\infty),X_{1,+})$ with $f_i\in C([0,\infty))$ for all $i\ge 1$;
		\item [(a2)] for each $i\ge 1$ and $t>0$, 
		\begin{equation*}
			 \sum_{j=1}^\infty a_{i+j,j} f_{i+j} f_j \in L^1((0,t))\,, \quad \sum_{j=1}^\infty a_{i,j} f_i f_j \in L^1((0,t))\,;
		\end{equation*}
		\item [(a3)] for each $i\ge 1$ and $t>0$, 
		\begin{equation*}
			f_i(t) = f_i^{in} + \int_0^t  \sum_{j=1}^\infty a_{i+j,j} f_{i+j}(s) f_j(s)\ \mathrm{d}s - \int_0^t \sum_{j=1}^\infty a_{i,j} f_i(s) f_j(s)\ \mathrm{d}s\,.
		\end{equation*}
	\end{itemize}
\end{definition}

\begin{theorem}\label{thm1}
	Assume that the kinetic coefficients $(a_{i,j})_{i,j\ge 1}$ satisfy~\eqref{a} and consider $f^{in}\in X_{1,+}$. Then there is at least one mild solution $f$ to~\eqref{rbak} in the sense of \Cref{defms} which satisfies additionally 
	\begin{equation}
		\sum_{j=1}^\infty \sum_{k=1}^\infty \min\{j,k\} a_{j,k} f_j f_k \in L^1((0,\infty)) \label{dm1}
	\end{equation}
	and
	\begin{equation}
		\|f(t)\|_1 + \int_0^t \sum_{j=1}^\infty \sum_{k=1}^\infty \min\{j,k\} a_{j,k} f_j(s) f_k(s)\ \mathrm{d}s = \|f^{in}\|_1\,, \qquad t>0\,. \label{dm2}
	\end{equation}
	In addition, if there is a non-decreasing sequence $(A_i)_{i\ge 1}$ of positive real numbers with $A_1\ge 1$ such that the kinetic coefficients $(a_{i,j})_{i,j\ge 1}$ and the initial condition $f^{in}$ satisfy
	\begin{equation}
		0\le a_{i,j} = a_{j,i} \le A_i A_j\,, \qquad i,j\ge 1\,, \label{aa}
	\end{equation} 
	and
	\begin{equation}
		M_A(f^{in}) := \sum_{i=1}^\infty A_i f_i^{in} < \infty\,, \label{ab}
	\end{equation} 
	then the mild solution $f$ constructed above satisfies
	\begin{equation}
		\sum_{i=m}^\infty A_i f_i(t) \le \sum_{i=m}^\infty A_i f_i^{in}\,, \qquad t\ge 0\,, \quad m\ge 1\,. \label{dm2a}
	\end{equation}
\end{theorem}

\Cref{thm1} shows that no growth condition or structure assumption is needed to ensure the existence of a global mild solution to~\eqref{rbak}. In addition, it provides the stability of the space of sequences satisfying~\eqref{ab} under the additional assumption~\eqref{aa} on the kinetic coefficients. Let us also emphasize that no growth condition is required on $(A_i)_{i\ge 1}$ in~\eqref{aa}.

We next turn to classical solutions and identify kinetic coefficients and initial conditions guaranteeing their existence.

\begin{theorem}\label{thm2}
	Assume that there is a non-decreasing sequence $(A_i)_{i\ge 1}$ of positive real numbers with $A_1\ge 1$ such that the kinetic coefficients $(a_{i,j})_{i,j\ge 1}$ satisfy~\eqref{aa} and consider $f^{in}\in X_{1,+}$ satisfying~\eqref{ab}. Then there is at least one classical solution $f=(f_i)_{i\ge 1}$ to~\eqref{rbak}; that is, $f_i\in C^1([0,\infty))$,
	\begin{equation}
		\sum_{j=1}^\infty a_{i+j,j} f_{i+j} f_j \in C([0,\infty))\,, \quad \sum_{j=1}^\infty a_{i,j} f_j \in C([0,\infty))\,, \label{dm1c}
	\end{equation}
	and~\eqref{rbak} is satisfied pointwisely for all $i\ge 1$. In addition, $f$ satisfies~\eqref{dm2}, as well as
	\begin{equation}
		M_A(f(t)) := \sum_{i=1}^\infty A_i f_i(t) \le M_A(f^{in})\,, \qquad t\ge 0\,. \label{mab}
	\end{equation}
\end{theorem}

\Cref{thm2} extends \cite[Theorem~3.1]{dCPS2012}, which corresponds to the choice $A_i=i \sqrt{K}$, $i\ge 1$, for some $K>0$. It is worth mentioning at this point that, given any kinetic coefficients $(a_{i,j})_{i,j\ge 1}$ satisfying~\eqref{a}, the assumption~\eqref{aa} is actually satisfied by the sequence $A^a = (A_{i}^a)_{i\ge 1}$ defined by
\begin{equation*}
	A_i^a := 1 + \max_{1\le j,k\le i} a_{j,k}\,, \qquad i\ge 1\,.
\end{equation*}
Therefore, the restrictive assumption in \Cref{thm2} is the tail behaviour~\eqref{ab} of $f^{in}$.

We supplement \Cref{thm2} with a uniqueness result, valid under a stronger assumption on the decay at infinity of the initial condition.

\begin{theorem}\label{thm3}
	Assume that there is a non-decreasing sequence $(A_i)_{i\ge 1}$ of positive real numbers with $A_1\ge 1$ such that the kinetic coefficients $(a_{i,j})_{i,j\ge 1}$ satisfy~\eqref{aa} and consider $f^{in}\in X_{1,+}$ satisfying
	\begin{equation}
		M_{A^2}(f^{in}) := \sum_{i=1}^\infty A_i^2 f_i^{in} < \infty\,, \label{a2b}
	\end{equation}
	Then there is a unique classical solution $f=(f_i)_{i\ge 1}$ to~\eqref{rbak} satisfying
	\begin{equation}
		M_{A^2}(f(t)) := \sum_{i=1}^\infty A_i^2 f_i(t) \le M_{A^2}(f^{in})\,, \qquad t\ge 0\,. \label{ma2b}
	\end{equation}
\end{theorem}

The evolution system~\eqref{rbak} bears some similarity with the celebrated Smoluchowski coagulation equation, which corresponds to the elementary reaction $P_i+P_j\to P_{i+j}$ and reads \cite{Smo1916}
\begin{equation}\label{sm}
	\begin{split}
		\frac{\mathrm{d}f_i}{\mathrm{d}t} & = \frac{1}{2} \sum_{j=1}^{i-1} a_{i-j,j} f_{i-j} f_j - \sum_{j=1}^\infty a_{i,j} f_i f_j\,, \qquad i\ge 1\,, \\
		f_i(0) & = f_i^{in}\,, \qquad i\ge 1\,.
	\end{split}
\end{equation}
Indeed, though describing different physical processes, their mathematical structure is similar. Both are countably infinite systems of quadratic differential equations, which are strongly coupled due to the infinite series involved in the reaction terms. Thus, not surprisingly, the techniques developed to study the well-posedness of~\eqref{sm} in the seminal paper \cite{BaCa1990} adapt well to~\eqref{rbak}. In particular, the same functional framework and similar assumptions on the rate coefficients are used in \cite{dCPS2012, Ver2024} to establish existence results for~\eqref{rbak}. Still, the dynamics of~\eqref{rbak} differs from that of~\eqref{sm}. Indeed, conservation or decrease of matter is expected for the latter, along with a monotone increase of superlinear moments (whenever finite), while mass and all superlinear moments are dissipated for the former throughout time evolution. In addition, Smoluchowski's coagulation equation~\eqref{sm} has no non-zero solution for rapidly increasing kinetic coefficients  such as $a_{i,j} =i^\alpha + j^\alpha$ with $\alpha>1$ \cite{CadC1992, vanD1987}. The outcome of \Cref{thm1} and \Cref{thm2} shows that no such phenomenon occurs for the system~\eqref{rbak}, a feature that can be explained by its dissipativity properties and we shall exploit thoroughly the latter in the analysis to be presented below. 

Specifically, \Cref{sec.2} is devoted to the proof of \Cref{thm1}, which relies on a compactness method and the approximation of~\eqref{rbak} by finite systems of ordinary differential equations as in \cite{BaCa1990, dCPS2012, Ver2024}. The cornerstone of the proof and the main contribution of this paper is \Cref{lem3}, which shows that the dissipation of the tails $\sum_{i=m}^\infty i f_i$ of the first moment controls the tails of the series on the right-hand side of~\eqref{rbak}. Besides guaranteeing the compactness of the approximate sequences, these estimates are instrumental in the derivation of the time evolution~\eqref{dm2} of the first moment. We next turn to the existence of classical solutions in \Cref{sec.3} which combines \Cref{thm1} and moment estimates. The uniqueness proof is provided in \Cref{sec.4} and both the assumptions on the initial condition in \Cref{thm3} and its proof are directly inspired from similar results for the coagulation-fragmentation equations, see \cite[Section~8.2.5]{BLL2019} and the references therein.   

\section{Existence: mild solutions}\label{sec.2}

We first recall the well-posedness of~\eqref{rbakf} established in \cite[Proposition~2.3]{dCPS2012}, along with a useful identity satisfied by solutions to~\eqref{rbakf}.

\begin{proposition}\label{prop1}
	Let $n\ge 3$ and $f^{in}\in X_{1,+}$. There is a unique solution $f^n=(f_i^n)_{1\le i \le n}\in C^1([0,\infty),[0,\infty)^n)$ to~\eqref{rbakf}. In addition, if $(\vartheta_i)_{i\ge 1}$ is a sequence of real numbers, then
	\begin{equation}
		\frac{\mathrm{d}}{\mathrm{d}t} \sum_{j=1}^n \vartheta_j f_j^n + \sum_{j=2}^n \sum_{k=1}^{j-1} (\vartheta_j - \vartheta_{j-k}) a_{j,k} f_j^n f_k^n + \sum_{j=1}^n \sum_{k=j}^n \vartheta_j a_{j,k} f_j^n f_k^n = 0\,. \label{p1}
	\end{equation}
\end{proposition}

From now on, $f^{in}\in X_{1,+}$ is given and, for each $n\ge 3$, $f^n=(f_i^n)_{1\le i \le n}$ denotes the corresponding solution to~\eqref{rbakf} provided by \Cref{prop1}.

\subsection{Compactness}\label{sec.2a}

We first draw several consequences of~\eqref{p1} and begin with the following observation when the sequence $(\vartheta_i)_{i\ge 1}$ is assumed to be non-negative and non-decreasing.

\begin{corollary}\label{cor2}
Let $(\vartheta_i)_{i\ge 1}$ be a non-negative and non-decreasing sequence. Then, for $n\ge 3$ and $t>0$,
\begin{subequations}\label{estg}
	\begin{align}
		0 & \le \sum_{j=1}^n \vartheta_j f_j^n(t) \le \sum_{i=1}^n \vartheta_j f_j^{in}\,, \label{estg1} \\
		0 & \le \int_0^t \sum_{j=2}^n \sum_{k=1}^{j-1} (\vartheta_j - \vartheta_{j-k}) a_{j,k} f_j^n(s) f_k^n(s)\ \mathrm{d}s \le \sum_{j=1}^n \vartheta_j f_j^{in}\,, \label{estg2} \\
		0 & \le \int_0^t \sum_{j=1}^n \sum_{k=j}^n \vartheta_j a_{j,k} f_j^n(s) f_k^n(s) \ \mathrm{d}s \le \sum_{j=1}^n \vartheta_j f_j^{in}\,. \label{estg3}
	\end{align}
\end{subequations}
\end{corollary}

\begin{proof}
	Since the sequence $(\vartheta_i)_{i\ge 1}$ is non-negative and non-decreasing, the three terms involved in the left-hand side of~\eqref{p1} are non-negative and \Cref{cor2} readily follows from~\eqref{p1} after integration with respect to time.
\end{proof}

We next use a specific choice of $(\vartheta_i)_{i\ge 1}$ in \Cref{cor2} to obtain the following estimates, which could also be derived directly from \cite[Proposition~2.3]{dCPS2012} with the same choice.

\begin{lemma}\label{lem3}
For $n\ge 3$, $m \ge 1$, and $t>0$,
\begin{subequations}\label{estm}
	\begin{align}
		& \sum_{j=m}^n j f_j^n(t) \le \sum_{j=m}^n j f_j^{in}\,, \label{estm1} \\
		& \int_0^t \sum_{j=m}^n \sum_{k=1}^n \min\{j,k\} a_{j,k} f_j^n(s) f_k^n(s)\ \mathrm{d}s \le 2\sum_{j=m}^n j f_j^{in}\,. \label{estm2}
	\end{align}
\end{subequations}
\end{lemma}

\begin{proof}
\Cref{lem3} readily follows from \Cref{cor2} with the choice
\begin{equation*}
	\vartheta_i = 0\,, \quad 1\le i \le m-1\,, \qquad \vartheta_i = i\,, \quad i\ge m\,.
\end{equation*}	
Indeed, this choice of the sequence $(\vartheta_i)_{i\ge 1}$ gives
\begin{align}
	\sum_{j=2}^n \sum_{k=1}^{j-1} (\vartheta_j - \vartheta_{j-k}) a_{j,k} f_j^n f_k^n & = \sum_{j=m}^n \sum_{k=j-m+1}^{j-1} j a_{j,k} f_j^n f_k^n + \sum_{j=m}^n \sum_{k=1}^{j-m} k a_{j,k} f_j^n f_k^n \nonumber \\
	& \ge \sum_{j=m}^n \sum_{k=1}^{j-1} \min\{j,k\} a_{j,k} f_j^n f_k^n \label{p2}
\end{align}
and
\begin{equation}
	\sum_{j=1}^n \sum_{k=j}^n \vartheta_j a_{j,k} f_j^n f_k^n = \sum_{j=m}^n \sum_{k=j}^n j a_{j,k} f_j^n f_k^n = \sum_{j=m}^n \sum_{k=j}^n \min\{j,k\} a_{j,k} f_j^n f_k^n\,, \label{p3}
\end{equation}
and we infer~\eqref{estm2} from~\eqref{estg2}, \eqref{estg3}, \eqref{p2}, and~\eqref{p3}.
\end{proof}

Applying \Cref{lem3} with $m=1$ provides the following estimates.

\begin{corollary}\label{cor4}
For $n\ge 3$ and $t>0$, 
\begin{subequations}\label{estu}
	\begin{align}
		& \sum_{j=1}^n j f_j^n(t) \le \sum_{j=1}^n j f_j^{in} \le \|f^{in}\|_1\,, \label{estu1} \\
		& \int_0^t \sum_{j=1}^n \sum_{k=1}^n \min\{j,k\} a_{j,k} f_j^n(s) f_k^n(s)\ \mathrm{d}s \le 2 \sum_{j=1}^n j f_j^{in} \le 2 \|f^{in}\|_1\,, \label{estu2} \\
		& \int_0^t \sum_{j=1}^n \left| \frac{\mathrm{d}f_j^n}{\mathrm{d}t}(s) \right|\ \mathrm{d}s \le 4 \| f^{in}\|_1\,. \label{estu3}
	\end{align}
\end{subequations}
\end{corollary}

\begin{proof}
	The bounds~\eqref{estu1} and~\eqref{estu2} are immediate consequences of~\eqref{estm1} and~\eqref{estm2} with $m=1$, respectively. We next infer from~\eqref{rbakf}, \eqref{estu2}, and the lower bound $\min\{j,k\}\ge 1$ for $j\ge 1$ and $k\ge 1$ that
	\begin{align*}
		\sum_{j=1}^n \left| \frac{\mathrm{d}f_j^n}{\mathrm{d}t} \right| & \le \sum_{j=1}^n \left[ \sum_{k=1}^{n-j} a_{j+k,k} f_{j+k}^n f_k^n + \sum_{k=1}^n a_{j,k} f_j^n f_k^n \right] \\
		& = \sum_{k=1}^n \sum_{j=1}^{n-k} a_{j+k,k} f_{j+k}^n f_k^n + \sum_{j=1}^n \sum_{k=1}^n a_{j,k} f_j^n f_k^n \\
		& \le \sum_{k=1}^n \sum_{j=k+1}^{n} a_{j,k} f_j^n f_k^n + 2 \|f^{in}\|_1 \le 4 \|f^{in}\|_1\,,
	\end{align*}
	and the proof of \Cref{cor4} is complete.
\end{proof}

After this preparation, we are in a position to state the main estimate of this section, which provides a control on the tails of the two infinite sums on the right-hand side of~\eqref{rbak1}.

\begin{proposition}\label{prop5}
	For $i\ge 1$, $t>0$, and $n\ge m+i \ge 2(i+1)$,
	\begin{align}
		& \int_0^t \sum_{j=m}^n a_{i,j} f_i^n(s) f_j^n(s)\ \mathrm{d}s \le 2 \sum_{j=m}^\infty j f_j^{in}\,, \label{p4} \\
		& \int_0^t \sum_{j=m}^{n-i} a_{i+j,j} f_{i+j}^n(s) f_j^n(s)\ \mathrm{d}s \le 2 \sum_{j=m}^\infty j f_j^{in}\,. \label{p5} 
	\end{align}
\end{proposition}

\begin{proof}
We first note that, for $j\in\{1,\ldots,n\}$,
\begin{equation*}
	1 \le i \le j-m \iff j \ge m+i \;\;\text{ and }\;\; j+1-m \le i \le n \iff j \le m + i -1\,.
\end{equation*}
Thanks to this observation and the choice $n\ge m+i$, 
\begin{equation*}
	\sum_{j=m+1}^n \sum_{k=1}^{j-m} k a_{j,k} f_j^n f_k^n \ge \sum_{j=m+i}^n \sum_{k=1}^{j-m} a_{j,k} f_j^n f_k^n \ge \sum_{j=m+i}^n a_{j,i} f_j^n f_i^n
\end{equation*}
and
\begin{equation*}
	\sum_{j=m}^n \sum_{k=j-m+1}^n \min\{j,k\} a_{j,k} f_j^n f_k^n \ge \sum_{j=m}^{m+i-1} \sum_{k=j-m+1}^n a_{j,k} f_j^n f_k^n \ge \sum_{j=m}^{m+i-1} a_{j,i} f_j^n f_i^n\,,
\end{equation*}
and we infer from~\eqref{estm2} and the above two inequalities that
\begin{align*}
	\int_0^t \sum_{j=m}^n a_{i,j} f_i^n(s) f_j^n(s)\ \mathrm{d}s & \le \int_0^t \sum_{j=m}^n \sum_{k=1}^n \min\{j,k\} a_{i,j} f_i^n(s) f_j^n(s)\ \mathrm{d}s \\
	& \le 2 \sum_{j=m}^n j f_j^{in} \le 2 \sum_{j=m}^\infty j f_j^{in}\,,
\end{align*}
hence~\eqref{p4}. Similarly, since $n\ge m+i$,
\begin{equation*}
	\sum_{j=m}^n \sum_{k=j-m+1}^n \min\{j,k\} a_{j,k} f_j^n f_k^n \ge \sum_{j=m}^{n-i} \sum_{k=j-m+1}^n a_{j,k} f_j^n f_k^n \ge \sum_{j=m}^{n-i} j a_{j,i+j} f_j^n f_{i+j}^n\,,
\end{equation*}
from which~\eqref{p5} readily follows due to~\eqref{estm2}.
\end{proof}

\subsection{Convergence}\label{sec.2b}

\begin{proof}[Proof of \Cref{thm1}]
Owing to \Cref{cor4} (and in particular~\eqref{estu1} and~\eqref{estu3}), we are in a position to apply Helly's selection principle \cite[Theorem~2.35]{Leo2009}, along with a diagonal process, to find a sequence $(n_l)_{l\ge 1}$, $n_l\to\infty$, and a sequence of functions $(f_j)_{j\ge 1}$ such that
\begin{equation}
	\lim_{l\to\infty} f_i^{n_l}(t) = f_i(t) \;\;\text{ for all }\;\; t\ge 0 \;\;\text{ and }\;\; i\ge 1\,. \label{co1}
\end{equation}

Let us now fix $i\ge 1$. For $t>0$, $m\ge i$, and $n_l\ge m+i$, it follows from~\eqref{estu2}, \eqref{co1}, and the Lebesgue dominated convergence theorem that
\begin{align*}
	\int_0^t \sum_{j=1}^m a_{i+j,j} f_{i+j}(s) f_j(s)\ \mathrm{d}s & \le \int_0^t \sum_{j=1}^m \sum_{k=1}^{m+i} a_{k,j} f_{k}(s) f_j(s)\ \mathrm{d}s \\
	& = \lim_{l\to\infty} \int_0^t \sum_{j=1}^m \sum_{k=1}^{m+i} a_{k,j} f_{k}^{n_l}(s) f_j^{n_l}(s)\ \mathrm{d}s \\
	& \le 2 \|f^{in}\|_1
\end{align*}
and
\begin{align*}
	\int_0^t \sum_{j=1}^m a_{i,j} f_{i}(s) f_j(s)\ \mathrm{d}s & \le \int_0^t \sum_{i=1}^m \sum_{j=1}^{m} a_{i,j} f_{i}(s) f_j(s)\ \mathrm{d}s \\
	& = \lim_{l\to\infty} \int_0^t \sum_{j=1}^m \sum_{k=1}^{m} a_{i,j} f_{i}^{n_l}(s) f_j^{n_l}(s)\ \mathrm{d}s \\
	& \le 2 \|f^{in}\|_1\,.
\end{align*}
Letting $m\to\infty$ in the above two inequalities and using Fatou's lemma lead us to
\begin{equation}
	\sum_{j=1}^\infty a_{i+j,j} f_{i+j} f_j\in L^1((0,t))\,, \quad \sum_{j=1}^\infty a_{i,j} f_{i} f_j\in L^1((0,t))\,. \label{co2}
\end{equation}
Similarly, we infer from~\eqref{estm2} that, for $n_l \ge r>m$,
\begin{equation*}
	\int_0^t \sum_{j=m}^r \sum_{k=1}^r \min\{j,k\} a_{j,k} f_j^{n_l}(s) f_k^{n_l}(s)\ \mathrm{d}s \le 2 \sum_{j=m}^\infty j f_j^{in}\,.
\end{equation*}
Taking the limit $l\to\infty$ and using~\eqref{co1} yield
\begin{equation*}
	\int_0^t \sum_{j=m}^r \sum_{k=1}^r \min\{j,k\} a_{j,k} f_j(s) f_k(s)\ \mathrm{d}s \le 2 \sum_{j=m}^\infty j f_j^{in}\,.
\end{equation*}
We then let $r\to\infty$ and deduce from Fatou's lemma that
\begin{equation}
	\int_0^t \sum_{j=m}^\infty \sum_{k=1}^\infty \min\{j,k\} a_{j,k} f_j(s) f_k(s)\ \mathrm{d}s \le 2 \sum_{j=m}^\infty j f_j^{in}\,. \label{dm1m}
\end{equation} 
We have thus shown that $f$ satisfies \Cref{defms}~(a2) and~\eqref{dm1}.

Next, for $t>0$ and $n_l\ge m+i \ge 2(i+1)$, we infer from~\eqref{p5} that
\begin{align*}
	& \int_0^t \left| \sum_{j=1}^{n_l-i} a_{i+j,j} f_{i+j}^{n_l}(s) f_j^{n_l}(s) - \sum_{j=1}^\infty a_{i+j,j} f_{i+j}(s) f_j(s) \right|\ \mathrm{d}s \\
	& \hspace{1cm} \le \int_0^t \sum_{j=1}^{m-1} a_{i+j,j} \left|  f_{i+j}^{n_l}(s) f_j^{n_l}(s) -  f_{i+j}(s) f_j(s) \right|\ \mathrm{d}s + \int_0^t \sum_{j=m}^{n-i} a_{i+j,j} f_{i+j}^{n_l}(s) f_j^{n_l}(s)\ \mathrm{d}s \\
	& \hspace{2cm} + \int_0^t \sum_{j=m}^\infty a_{i+j,j} f_{i+j}(s) f_j(s)\ \mathrm{d}s \\
	& \hspace{1cm} \le \int_0^t \sum_{j=1}^{m-1} a_{i+j,j} \left|  f_{i+j}^{n_l}(s) f_j^{n_l}(s) -  f_{i+j}(s) f_j(s) \right|\ \mathrm{d}s + 2 \sum_{j=m}^{\infty} j f_j^{in} \\
	& \hspace{2cm} + \int_0^t \sum_{j=m}^\infty a_{i+j,j} f_{i+j}(s) f_j(s)\ \mathrm{d}s\,.
\end{align*}
Thanks to~\eqref{estu1}, \eqref{co1}, and the Lebesgue dominated convergence theorem, we may take the limit $\l\to\infty$ in the above inequality and find
\begin{align*}
	& \limsup_{l\to\infty} \int_0^t \left| \sum_{j=1}^{n_l-i} a_{i+j,j} f_{i+j}^{n_l}(s) f_j^{n_l}(s) - \sum_{j=1}^\infty a_{i+j,j} f_{i+j}(s) f_j(s) \right|\ \mathrm{d}s \\
	& \hspace{1cm} \le 2 \sum_{j=m}^{\infty} j f_j^{in} + \int_0^t \sum_{j=m}^\infty a_{i+j,j} f_{i+j}(s) f_j(s)\ \mathrm{d}s\,.
\end{align*}
Due to $f^{in}\in X_{1,+}$ and~\eqref{co2}, we let $m\to\infty$ in the above inequality to conclude that
\begin{equation}
	\lim_{l\to\infty} \int_0^t \sum_{j=1}^{n_l-i} a_{i+j,j} f_{i+j}^{n_l}(s) f_j^{n_l}(s)\ \mathrm{d}s = \int_0^t \sum_{j=1}^\infty a_{i+j,j} f_{i+j}(s) f_j(s)\ \mathrm{d}s\,. \label{co3}
\end{equation}
A similar argument, based on~\eqref{p4} instead of~\eqref{p5}, gives
\begin{equation}
	\lim_{l\to\infty} \int_0^t \sum_{j=1}^{n_l} a_{i,j} f_{i}^{n_l}(s) f_j^{n_l}(s)\ \mathrm{d}s = \int_0^t \sum_{j=1}^\infty a_{i,j} f_{i}(s) f_j(s)\ \mathrm{d}s\,. \label{co4}
\end{equation}
Since $f_i^{n_l}$ is a classical solution to~\eqref{rbakf} on $[0,\infty)$, it satisfies 
\begin{equation*}
	f_i^{n_l}(t) = f_i^{in} +  \int_0^t \sum_{j=1}^{n_l-i} a_{i+j,j} f_{i+j}^{n_l}(s) f_j^{n_l}(s)\ \mathrm{d}s - \int_0^t \sum_{j=1}^{n_l} a_{i,j} f_{i}^{n_l}(s) f_j^{n_l}(s)\ \mathrm{d}s
\end{equation*}
and \eqref{co1}, \eqref{co3}, and~\eqref{co4} allow us to take the limit $l\to\infty$ in the above identity and conclude that $f$ satisfies \Cref{defms}~(a3). In particular, the latter and~\eqref{co2} entail that $f_i\in C([0,\infty))$.

Now, the boundedness of $f$ in $X_{1,+}$ is an immediate consequence of~\eqref{estu1} and~\eqref{co1}. Collecting the outcome of the above analysis, we have established that $f$ is a mild solution to~\eqref{rbak} in the sense of \Cref{defms} and satisfies~\eqref{dm1}. 

We are left with proving~\eqref{dm2}. To this end, we infer from~\eqref{p1} with $\vartheta_j=j$, $j\ge 1$, that, for $l\ge 1$,
\begin{equation}
	\sum_{j=1}^{n_l} j f_j^{n_l}(t) + \int_0^t \sum_{j=1}^{n_l} \sum_{k=1}^{n_l} \min\{j,k\} a_{j,k} f_j^{n_l}(s) f_k^{n_l}(s)\ \mathrm{d}s =  \sum_{j=1}^{n_l} j f_j^{in}\,. \label{dm3}
\end{equation}
On the one hand, we readily infer from~\eqref{estm1} that, for $n_l>m$, 
\begin{align*}
	\left| \sum_{j=1}^{n_l} j f_j^{n_l}(t) - \sum_{j=1}^\infty j f_j(t) \right| & \le \sum_{j=1}^m j \big| f_j^{n_l}(t) - f_j(t) \big| + \sum_{j=m}^{n_l} j f_j^{n_l}(t) + \sum_{j=m}^\infty j f_j(t) \\
	& \le \sum_{j=1}^m j \big| f_j^{n_l}(t) - f_j(t) \big| + \sum_{j=m}^\infty j f_j^{in} + \sum_{j=m}^\infty j f_j(t) \,.
\end{align*}
Owing to~\eqref{co1}, we may pass to the limit $l\to\infty$ in the above inequality and obtain
\begin{equation*}
	\limsup_{l\to\infty} \left| \sum_{j=1}^{n_l} j f_j^{n_l}(t) - \sum_{j=1}^\infty j f_j(t) \right| \le \sum_{j=m}^\infty j f_j^{in} + \sum_{j=m}^\infty j f_j(t) \,.
\end{equation*}
Since both $f(t)$ and $f^{in}$ belong to $X_{1,+}$, we let $m\to\infty$ to conclude that
\begin{equation}
	\lim_{l\to\infty} \sum_{j=1}^{n_l} j f_j^{n_l}(t) = \sum_{j=1}^\infty j f_j(t) \,. \label{dm4} 
\end{equation} 
On the other hand, it follows from~\eqref{estm2} and~\eqref{dm1m} that, for $n_l> m\ge 3$,
\begin{align*}
	& \int_0^t \left| \sum_{j=1}^{n_l} \sum_{k=1}^{n_l} \min\{j,k\} a_{j,k} f_j^{n_l}(s) f_k^{n_l}(s) -  \sum_{j=1}^{\infty} \sum_{k=1}^{\infty} \min\{j,k\} a_{j,k} f_j(s) f_k(s) \right|\ \mathrm{d}s \\
	& \qquad \le \int_0^t \sum_{j=1}^{m-1} \sum_{k=1}^{m-1} \min\{j,k\} a_{j,k} \big| f_j^{n_l}(s) f_k^{n_l}(s) - f_j(s) f_k(s) \big|\ \mathrm{d}s \\
	& \qquad\qquad + \int_0^t \left[ \sum_{j=1}^{m-1} \sum_{k=m}^{n_l} \min\{j,k\} a_{j,k} f_j^{n_l}(s) f_k^{n_l}(s) + \sum_{j=1}^{m-1} \sum_{k=m}^\infty  \min\{j,k\} a_{j,k} f_j(s) f_k(s) \right] \ \mathrm{d}s \\
	& \qquad\qquad + \int_0^t \left[ \sum_{j=m}^{n_l} \sum_{k=1}^{n_l} \min\{j,k\} a_{j,k} f_j^{n_l}(s) f_k^{n_l}(s) + \sum_{j=m}^{\infty} \sum_{k=1}^\infty  \min\{j,k\} a_{j,k} f_j(s) f_k(s) \right] \ \mathrm{d}s \\
	& \qquad\le \int_0^t \sum_{j=1}^{m-1} \sum_{k=1}^{m-1} \min\{j,k\} a_{j,k} \big| f_j^{n_l}(s) f_k^{n_l}(s) - f_j(s) f_k(s) \big|\ \mathrm{d}s \\
	& \qquad\qquad + \int_0^t \left[ \sum_{j=m}^{n_l} \sum_{k=1}^{m-1} \min\{j,k\} a_{j,k} f_j^{n_l}(s) f_k^{n_l}(s) + \sum_{j=m}^\infty \sum_{k=1}^{m-1}  \min\{j,k\} a_{j,k} f_j(s) f_k(s) \right] \ \mathrm{d}s \\
	& \qquad\qquad + \int_0^t \left[ \sum_{j=m}^{n_l} \sum_{k=1}^{n_l} \min\{j,k\} a_{j,k} f_j^{n_l}(s) f_k^{n_l}(s) + \sum_{j=m}^{\infty} \sum_{k=1}^\infty  \min\{j,k\} a_{j,k} f_j(s) f_k(s) \right] \ \mathrm{d}s \\
	& \qquad\le \int_0^t \sum_{j=1}^{m-1} \sum_{k=1}^{m-1} \min\{j,k\} a_{j,k} \big| f_j^{n_l}(s) f_k^{n_l}(s) - f_j(s) f_k(s) \big|\ \mathrm{d}s + 8 \sum_{j=m}^\infty j f_j^{in}\,.
\end{align*}
Thanks to~\eqref{estu1}, \eqref{co1}, and Lebesgue's dominated convergence theorem, we may pass to the limit $l\to\infty$ in the above estimate and find
\begin{align*}
	& \limsup_{l\to\infty} \int_0^t \left| \sum_{j=1}^{n_l} \sum_{k=1}^{n_l} \min\{j,k\} a_{j,k} f_j^{n_l}(s) f_k^{n_l}(s) -  \sum_{j=1}^{\infty} \sum_{k=1}^{\infty} \min\{j,k\} a_{j,k} f_j(s) f_k(s) \right|\ \mathrm{d}s \\
	& \hspace{10cm} \le 8 \sum_{j=m}^\infty j f_j^{in}\,.
\end{align*}
We then let $m\to\infty$ and deduce that, since $f^{in}\in X_{1,+}$,
\begin{equation}
	\lim_{l\to\infty} \int_0^t \sum_{j=1}^{n_l} \sum_{k=1}^{n_l} \min\{j,k\} a_{j,k} f_j^{n_l}(s) f_k^{n_l}(s)\ \mathrm{d}s = \int_0^t \sum_{j=1}^{\infty} \sum_{k=1}^{\infty} \min\{j,k\} a_{j,k} f_j(s) f_k(s)\ \mathrm{d}s\,. \label{dm5}
\end{equation}
Gathering~\eqref{dm4} and~\eqref{dm5} allows us to take the limit $l\to\infty$ in~\eqref{dm3} and thereby derive~\eqref{dm2}, thus completing the proof of \Cref{thm1}.

We finally assume that the kinetic coefficients $(a_{i,j})_{i,j\ge 1}$ and the initial condition $f^{in}$ satisfy~\eqref{aa} and~\eqref{ab}, respectively. Let $m\ge 1$ and $t\ge 0$. We infer from~\eqref{estg1} (with $\vartheta_i=A_i$ for $i\ge m$ and $\vartheta_i=0$ for $1\le i\le m-1$) that
\begin{equation*}
	\sum_{i=m}^{r} A_i f_i^{n_l}(t) \le \sum_{i=m}^{n_l} A_i f_i^{n_l}(t) \le \sum_{i=m}^{n_l} A_i f_i^{in} \le \sum_{i=m}^\infty A_i f_i^{in} 
\end{equation*} 
for $l\ge 1$ large enough such that $n_l>r>m$. Thanks to~\eqref{co1}, we may first let $l\to\infty$ and then $r\to\infty$ in the above inequality to obtain~\eqref{dm2a} and complete the proof.
\end{proof}

\section{Classical solutions: existence}\label{sec.3}

\begin{proof}[Proof of \Cref{thm2}]
We now assume that the kinetic coefficients $(a_{i,j})_{i,j\ge 1}$ and the initial condition $f^{in}$ satisfy~\eqref{aa} and~\eqref{ab}, respectively. It follows from \Cref{thm1} that~\eqref{rbak} has a mild solution $f$ which satisfies~\eqref{dm2a} and we shall show that this last property implies the continuity properties~\eqref{dm1c} and the $C^1$-regularity of $f$. Indeed, for $(t,s)\in [0,\infty)^2$ and $m> i\ge 1$, we infer from~\eqref{dm2}, \eqref{aa}, and \eqref{dm2a} that 
\begin{align*}
	& \left| \sum_{j=1}^\infty a_{i,j} f_i(s)f_j(s) - \sum_{j=1}^\infty a_{i,j} f_i(t)f_j(t) \right| \le \sum_{j=1}^\infty a_{i,j} \big| (f_if_j)(t) - (f_if_j)(s)\big| \\
	& \hspace{2cm} \le \sum_{j=1}^{m-1} a_{i,j} \big| (f_if_j)(t) - (f_if_j)(s)\big| +  A_i \sum_{j=m}^\infty A_j \big[ (f_if_j)(t) + (f_if_j)(s) \big] \\
	& \hspace{2cm} \le \sum_{j=1}^{m-1} a_{i,j} \big| (f_if_j)(t) - (f_if_j)(s)\big| +  2A_i \|f^{in}\|_1 \sum_{j=m}^\infty A_j f_j^{in}\,.
\end{align*}
Owing to the continuity of $f_j$ for all $j\ge 1$, 
\begin{equation*}
	\limsup_{s\to t} \left| \sum_{j=1}^\infty a_{i,j} f_i(s)f_j(s) - \sum_{j=1}^\infty a_{i,j} f_i(t)f_j(t) \right| \le 2A_i \|f^{in}\|_1 \sum_{j=m}^\infty A_j f_j^{in}\,,
\end{equation*}
and, thanks to~\eqref{ab}, we may let $m\to\infty$ in the above inequality to conclude that
\begin{equation}
	\lim_{s\to t} \sum_{j=1}^\infty a_{i,j} f_i(s)f_j(s) = \sum_{j=1}^\infty a_{i,j} f_i(t)f_j(t)\,. \label{c04}
\end{equation}
Similarly, let $(t,s)\in [0,\infty)^2$ and $m> i\ge 1$. By~\eqref{aa}, \eqref{ab}, and \eqref{dm2a},
\begin{align*}
	& \left| \sum_{j=1}^\infty a_{i+j,j} f_{i+j}(s)f_j(s) - \sum_{j=1}^\infty a_{i+j,j} f_{i+j}(t)f_j(t) \right| \le \sum_{j=1}^\infty a_{i+j,j} \big| (f_{i+j}f_j)(t) - (f_{i+j}f_j)(s)\big| \\
	& \hspace{2cm} \le \sum_{j=1}^{m-1} a_{i+1,j} \big| (f_{i+j}f_j)(t) - (f_{i+j}f_j)(s)\big| +  \sum_{j=m}^\infty A_{i+j} A_j \big[ (f_{i+j}f_j)(t) + (f_{i+j}f_j)(s) \big] \\
	& \hspace{2cm} \le \sum_{j=1}^{m-1} a_{i+j,j} \big| (f_{i+j}f_j)(t) - (f_{i+j}f_j)(s)\big| + \left( M_A(f(t)) \sum_{j=m}^\infty A_j f_j(t) + M_A(f(s)) \sum_{j=m}^\infty A_j f_j(s) \right) \\
	&  \hspace{2cm} \le \sum_{j=1}^{m-1} a_{i+j,j} \big| (f_{i+j}f_j)(t) - (f_{i+j}f_j)(s)\big| + 2 M_A(f^{in}) \sum_{j=m}^\infty A_j f_j^{in}\,.
\end{align*}
We next proceed as in the proof of~\eqref{c04} to obtain
\begin{equation}
	\lim_{s\to t} \sum_{j=1}^\infty a_{i+j,j} f_{i+j}(s)f_j(s) = \sum_{j=1}^\infty a_{i+j,j} f_{i+j}(t)f_j(t)\,. \label{c05}
\end{equation}
The continuity~\eqref{dm1c} is then an immediate consequence of~\eqref{c04} and~\eqref{c05} and we combine \Cref{defms}~(a3) and~\eqref{dm1c} to conclude that $f_i\in C^1([0,\infty))$ for each $i\ge 1$. Finally, the bound~\eqref{mab} readily follows from~\eqref{dm2a} with $m=1$.
\end{proof}

\section{Classical solutions: uniqueness}\label{sec.4}

\begin{proof}[Proof of \Cref{thm3}]	First, since $A_i^2\ge A_1 A_i \ge A_i$ for $i\ge 1$, the sequence $(A_i^2)_{i\ge 1}$ is a non-decreasing sequence of positive real numbers with $A_1^2\ge 1$ which satisfies~\eqref{aa} and we infer from~\eqref{a2b} and \Cref{thm2} that there is at least one classical solution $f$ to~\eqref{rbak} satisfying~\eqref{ma2b}.
	
As for uniqueness, we proceed along the lines of the proof of \cite[Proposition~5.1]{dCPS2012}, using the sequence $(A_i)_{i\ge 1}$ as a weight instead of $(i^\alpha)_{i\ge 1}$. Specifically, let $f$ and $g$ be two classical solutions to~\eqref{rbak} satisfying~\eqref{ma2b} and set $E=f-g$. Then, for $i\ge 1$, $E_i$ solves
\begin{equation*}
	\frac{\mathrm{d}E_i}{\mathrm{d}t} = \sum_{j=1}^\infty a_{i+j,j} \left( f_j E_{i+j} + g_{i+j} E_j \right) - \sum_{j=1}^\infty a_{i,j} \left( f_j E_i + g_i E_j \right),
\end{equation*} 
from which we deduce that
\begin{align*}
	\frac{\mathrm{d}}{\mathrm{d}t} \sum_{i=1}^\infty A_i |E_i| & = \sum_{i=1}^\infty  \sum_{j=1}^\infty A_i a_{i+j,j} \mathrm{sign}(E_i) \left( f_j E_{i+j} + g_{i+j} E_j \right) \\
	& \qquad - \sum_{i=1}^\infty \sum_{j=1}^\infty A_i a_{i,j} \mathrm{sign}(E_i) \left( f_j E_i + g_i E_j \right) \\
	& \le \sum_{i=1}^\infty \sum_{j=i+1}^\infty A_{j-i} a_{i,j} \left( f_i |E_{j}| + g_{j} |E_i| \right)  + \sum_{i=1}^\infty \sum_{j=1}^\infty A_i a_{i,j} g_i |E_j| \\
	& \qquad - \sum_{i=1}^\infty \sum_{j=1}^\infty A_i a_{i,j} f_j |E_i| \,.
\end{align*}
Using the symmetry~\eqref{a} of $(a_{i,j})_{i,j\ge 1}$ and the monotonicity of $(A_i)_{i\ge 1}$, we further obtain
\begin{align*}
	\frac{\mathrm{d}}{\mathrm{d}t} \sum_{i=1}^\infty A_i |E_i| & \le  \sum_{i=1}^\infty \sum_{j=i+1}^\infty A_{j} a_{i,j} f_i |E_{j}| + 2 \sum_{i=1}^\infty \sum_{j=1}^\infty A_j a_{i,j} g_j |E_i| - \sum_{i=1}^\infty \sum_{j=1}^\infty A_j a_{i,j} f_i |E_j|\\
	& \le 2 \sum_{i=1}^\infty \sum_{j=1}^\infty A_j a_{i,j} g_j |E_i|\,.
\end{align*}
We finally infer from~\eqref{aa}, \eqref{ma2b}, and the above inequality that
\begin{equation*}
	\frac{\mathrm{d}}{\mathrm{d}t} \sum_{i=1}^\infty A_i |E_i| \le 2 \sum_{i=1}^\infty \sum_{j=1}^\infty A_j^2 A_{i} g_j |E_i| \le 2 M_{A^2}(f^{in}) \sum_{i=1}^\infty A_i |E_i|\,, 
\end{equation*}
and Gronwall's lemma completes the proof.
\end{proof}

\begin{remark}
	Let $R>0$. It actually follows from \Cref{thm3} and its proof that, when the kinetic coefficients $(a_{i,j})_{i,j\ge 1}$ satisfy~\eqref{aa}, the system~\eqref{rbak} generates a dynamical system in the complete metric space $\{z = (z_i)_{i\ge 1}\in X_{1,+}\ :\ \|z\|_1+ M_{A^2}(z) \le R\}$ endowed with the metric induced by the norm $\|z\|_A := \sum_{i=1}^\infty A_i |z_i|$.
\end{remark}

\section*{Acknowledgments}

I thank Ankik Kumar Giri for pointing out the Redner-ben-Avraham-Kahng cluster system to me, as well as the referee for valuable remarks.

\noindent For the purpose of Open Access, a CC-BY public copyright licence has been applied by the author to the present document and will be applied to all subsequent versions up to the Author Accepted Manuscript arising from this submission.

\bibliographystyle{siam}
\bibliography{RBAK}

\begin{thebibliography}{1}

\bibitem{BaCa1990}
{\sc J.~M. Ball and J.~Carr}, {\em The discrete coagulation-fragmentation
  equations: {E}xistence, uniqueness, and density conservation}, J. Statist.
  Phys., 61 (1990), pp.~203--234.

\bibitem{BLL2019}
{\sc J.~Banasiak, W.~Lamb, and {\relax Ph}.~Lauren{\c{c}}ot}, {\em Analytic
  methods for coagulation-fragmentation models. {Volume} {II}}, Monogr. Res.
  Notes Math., Boca Raton, FL: CRC Press, 2020.

\bibitem{CadC1992}
{\sc J.~Carr and F.~P. da~Costa}, {\em Instantaneous gelation in coagulation
  dynamics}, Z. Angew. Math. Phys., 43 (1992), pp.~974--983.

\bibitem{dCPS2012}
{\sc F.~P. da~Costa, J.~T. Pinto, and R.~Sasportes}, {\em The
  {R}edner-{B}en-{A}vraham-{K}ahng cluster system}, S\~{a}o Paulo J. Math.
  Sci., 6 (2012), pp.~171--201.

\bibitem{Leo2009}
{\sc G.~Leoni}, {\em A first course in {Sobolev} spaces}, vol.~105 of Grad.
  Stud. Math., Providence, RI: American Mathematical Society (AMS), 2009.

\bibitem{RbAK1987}
{\sc S.~Redner, D.~ben Avraham, and B.~Kahng}, {\em Kinetics of 'cluster
  eating'}, Journal of Physics A: Mathematical and General, 20 (1987),
  pp.~1231--1238.

\bibitem{Smo1916}
{\sc M.~v. {Smoluchowski}}, {\em Drei {V}ortr{\"a}ge \"uber {D}iffusion,
  {B}rownsche {B}ewegung und {K}oagulation von {K}olloidteilchen}, Physik.
  Zeitschr., 17 (1916), pp.~557--571, 585--599.

\bibitem{vanD1987}
{\sc P.~G.~J. van Dongen}, {\em On the possible occurrence of instantaneous
  gelation in {S}moluchowski's coagulation equation}, J. Phys. A, 20 (1987),
  pp.~1889--1904.

\bibitem{Ver2024}
{\sc P.~Verma}, {\em Global existence and mass decay analysis of solutions to
  the discrete {R}edner-{B}en-{A}vraham-{K}ahng coagulation model}, 2023.
\newblock arXiv:2307.08868v2.

\end{thebibliography}

\end{document}